\documentclass[10pt]{article}

\usepackage{amssymb}
\usepackage{amsfonts}
\usepackage{amsmath}
\usepackage{amssymb}
\usepackage{graphicx}
\usepackage{caption}
\usepackage{float}
\usepackage[utf8]{inputenc}
\usepackage[english]{babel}
\usepackage{amsthm}

\newtheorem{theorem}{Theorem}
\newtheorem{lemma}[theorem]{Lemma}
\newtheorem{conjecture}[theorem]{Conjecture}
\newtheorem{remark}{Remark}

\begin{document}
			
		\title{Exponential Generalised Network Descriptors}
		

		\author{Suzana Antunovi\'c (1), Ton\'ci Kokan (2), Tanja Vojkovi\'c (3) \footnote{Corresponding author: tanja@pmfst.hr},\\ Damir Vuki\v{c}evi\'c (3)}

			
		\maketitle
		
		\noindent (1) Faculty of Civil Engineering, Architecture and Geodesy, Split, Croatia\\
		(2) Department of Mathematics, Faculty of Science, Zagreb, Croatia\\
		(3) Department of Mathematics, Faculty of Science, Split, Croatia
		
		\begin{abstract}
			In communication networks theory the concepts of networkness and network
			surplus have recently been defined. Together with transmission and
			betweenness centrality, they were based on the assumption of equal
			communication between vertices. Generalised versions of these four
			descriptors were presented, taking into account that
			communication between vertices $u$ and $v$ is decreasing as the distance between  them is increasing.
			Therefore, we weight the quantity of communication by  $\lambda^{d(u,v)}$ where $\lambda \in  \left\langle0,1 \right\rangle$. 
			Extremal values of these descriptors are analysed.
		\end{abstract}
	
	Keywords:
		graph theory, complex networks, network descriptors, transmission, centrality 
		
		68R10, 05C35, 94C15 

\section{Introduction}
\label{}

Complex networks are extensively used to model objects and their relations \cite{BI2}, \cite{BI1}. Throughout this paper we consider the  representation of a complex network as a simple connected graph $G=\left(V,E\right)$ and use standard graph--theoretical terminology \cite{bollobas}.

Betweenness centrality is one of key concepts in the study of complex networks \cite{BI6}, \cite{freeman78} and it can be efficiently calculated by Girvan--Newman algorithm \cite{BI4b}, \cite{BI4}.

For an edge $uv$, edge betweenness $b(uv)$ is defined in the following way:
$$
b(uv) = \sum_{\{k,l\}\in {\binom{V}{2}}}\frac{s_{uv}^{kl}}{s^{kl}},
$$
where $s_{uv}^{kl}$ is the number of shortest paths between vertices $k$ and $l$ that pass through the edge $uv$ and $s^{kl}$ is the total number of shortest paths between $k$ and $l$.

Betweenness centrality $c(u)$ of a vertex $u$ is sum of edge betweennesses of all edges incident to $u$:
$$
c\left( u\right) =\sum_{v\in [u]}b\left( uv\right),
$$
where $[u]$ is the set of neighbours of vertex $u$. Note that measure is closely related to, yet different from Freeman's betweenness centrality defined in \cite{BI7}:
$$
b(u) = c(u) - n + 1.
$$

In the context of the communication networks, betweenness centrality $c\left( u\right)$ can be interpreted as the quantity of communication processed by a vertex $u$ as stated in \cite{BI8}. On the other hand, transmission of the vertex $u$ defined as
$$t(u) =\sum_{v\in V}d(u,v)$$
where $d(u,v)$ is the distance between vertices $u$ and $v$, can be interpreted as the cost of the vertex to the network \cite{BI8}. 

Network surplus of the vertex $u$ ("added value'' to the network provided by vertex $u$) is defined by $\nu \left( u\right) =c\left( u\right) -t\left(u\right)$. Another way to measure productivity of vertex $u$ is its networkness defined in \cite{BI8} by $N\left( u\right) =c\left( u\right) /t\left(u\right)$ .
Note that interpretation of the betweenness centrality as the amount of information processed by the vertex $u$ assumes that the quantity of the information exchanged by any two vertices is equal. This was amended in \cite{BV} by weighting the amount of communication by $d\left( u,v\right) ^{\lambda}$ for some $\lambda < 0$, generalising the case $\lambda = -1$ introduced in \cite{borgatti}. Now we consider network descriptors based on the assumption that the amount of communication decreases as the distance between two vertices increases. Moreover, we assume that this amount is proportional to $\lambda^{d(u,v)}$ where $\lambda \in \left\langle0,1 \right\rangle$. We define:
$$
t_{\lambda}^{e}(u) =\sum_{v\in V\setminus\{u\}}d(u,v)\cdot \lambda^{d(u,v)}, 
$$

$$
b_{\lambda}^{e}(uv) = \sum_{\{k,l\}\in {\binom{V}{2}}}\frac{s_{uv}^{kl}}{s^{kl}}\cdot \lambda^{d(u,v)},
$$

$$
c_{\lambda}^{e}(u) =\sum_{v\in [u]}{\sum }b_{\lambda}^{e}(uv)
$$

Furthermore, we define:
$$
N_{\lambda}^{e}(u) =\frac{c_{\lambda}^{e}(u) }{t_{\lambda}^{e}(u) },
$$
$$
\nu _{\lambda}^{e}(u)  =c_{\lambda}^{e}(u) -t_{\lambda}^{e}(u).
$$

Analogously as in \cite{BI8} we define: 
\begin{center}
	\begin{tabular}[c]{lll}
		& & \\
		$mc_{\lambda}^{e}(G) = \min \left\{ c_{\lambda}^{e}(u): u\in V \right\}$ & & \\[2mm]
		$Mc_{\lambda}^{e}(G) = \max \left\{ c_{\lambda}^{e}(u): u\in V \right\} $ &  & \\[4mm]
		$mt_{\lambda}^{e}(G)= \min \left\{ t_{\lambda}^{e}(u): u\in V \right\} $ & & \\[2mm]
		$Mt_{\lambda}^{e}(G) = \max \left\{ t_{\lambda}^{e}(u): u\in V \right\} $ & &\\[4mm] 
		$mN_{\lambda}^{e}(G) = \min \left\{ N_{\lambda}^{e}(u): u\in V \right\} $ & & \\ [2mm]
		$ MN_{\lambda}^{e}(G) = \max \left\{ N_{\lambda}^{e}(u): u\in V \right\} $ & & \\ [4mm]
		$ m\nu _{\lambda}^{e}(G) = \min \left\{ \nu _{\lambda}^{e}(u): u\in V \right\}$ & & \\ [2mm]
		$ M\nu _{\lambda}^{e}(G) = \max \left\{ \nu _{\lambda}^{e}(u): u\in V \right\}$ & & \\ [2mm]
	\end{tabular}
\end{center}
and we are interested in finding the lower and upper bounds of these values for
all $\lambda \in \left\langle 0,1 \right\rangle$.
Our results can be summarized in the following way:
\vskip 0.3cm
\label{Tablica}
\begin{center}
	\begin{tabular}{|l||c|c|}
		\hline
		Descriptor & \multicolumn{2}{|c|}{$\lambda \in \left\langle 0,1\right\rangle 
			$} \\ \hline
		& Lower bound & Upper bound \\ \hline
		\multicolumn{1}{|c||}{$mt_{\lambda }^{e}$} & broom (starting vertex) & complete
		graph \textbf{*} \\ \hline
		\multicolumn{1}{|c||}{} & $A_n$ & $%
		(n-1)\lambda $ \\ \hline
		\multicolumn{1}{|c||}{$Mt_{\lambda }^{e}$} & \emph{open problem} & broom (starting vertex) \\ \hline
		&  & $B_n$ \\ \hline
		\multicolumn{1}{|c||}{$mc_{\lambda }^{e}$} & path (end vertices) & complete
		graph \textbf{*}\\ \hline
		& $\frac{\lambda^D-\lambda}{\lambda -1}$ & $(n-1)\lambda $ \\ \hline
		\multicolumn{1}{|c||}{$Mc_{\lambda }^{e}$} & \emph{open problem} & star
		(center) \\ \hline
		&  & $(n-1)\left[ \lambda +\frac{1}{2}(n-2)\lambda ^{2}\right] $ \\ \hline
		\multicolumn{1}{|c||}{$mN_{\lambda }^{e}$} & broom (starting vertex) & 
		vertex-transitive graph \\ \hline
		& $C_n$ & $1$ \\ \hline
		\multicolumn{1}{|c||}{$MN_{\lambda }^{e}$} & vertex-transitive graph & star
		(center) \\ \hline
		& $1$ & $\frac{1}{2}(n-2)\lambda +1$ \\ \hline
		\multicolumn{1}{|c||}{$m\nu _{\lambda }^{e}$} & broom (starting vertex) & vertex-transitive graph
		\\ \hline
		& $D_n$& $0$ \\ \hline
		\multicolumn{1}{|c||}{$M\nu _{\lambda }^{e}$} & vertex-transitive graph & 
		star (center) \\ \hline
		& $0$ & $\frac{1}{2}(n-1)(n-2)\lambda ^{2}$ \\ \hline
	\end{tabular}
	\captionof{table}{Extremal values of exponential generalised network descriptors}
\end{center}
\vskip 0.3cm
\
The terms $A_n, B_n, C_n$ and $D_n$ from Table \ref{Tablica} represent as follows:

\vskip 0,5cm

$A_n=\underset{1\leqslant D\leqslant n-1}{\min} {\frac{\lambda\left[1 -(D+1) \lambda  ^D+D\lambda ^{D+1}\right]}{(\lambda -1)^2}+(n-D-1)D\cdot \lambda ^{D}}$,

\vskip 0,5cm

$B_n=\underset{1\leqslant D\leqslant n-1}{\max} {\frac{\lambda\left[1 -(D+1) \lambda  ^D+D\lambda ^{D+1}\right]}{(\lambda -1)^2}+(n-D-1)D\cdot \lambda ^{D}}$,

\vskip 0,5cm

$C_n=\underset{1\leqslant D\leqslant n-1}{\min} {\frac{\lambda^D+\frac{1}{n-D}\left(\frac{\lambda^D-\lambda}{\lambda -1}\right)}{D\lambda^D+\frac{1}{n-D}\left[\frac{\lambda -D \lambda  ^D+(D-1)\lambda ^{D+1}}{(\lambda -1)^2}\right]}}$,

\vskip 0,5cm

$D_n=\underset{1\leqslant D\leqslant n-1}{\min} {\frac{\lambda\left[D \lambda^{D}-\lambda-(D-1)\lambda ^{D+1}\right]}{(\lambda -1)^2}+(n-D-1)(\lambda^D- D\cdot \lambda ^{D})}$.

\vskip 0,5cm

\begin{remark}
	Upper bounds marked \textbf{(*)} were stated and proven for $\lambda \in \left\langle 0, \frac{1}{2} \right\rangle$. They do not hold in general case. 
\end{remark}

\section{Connection between $t_{\lambda }^{e}$ and $c_{\lambda }^{e}$}

As in the papers \cite{BV} and \cite{BI8}, $t_{\lambda }^{e}$ can be considered as the cost of the vertex to the network and $c_{\lambda }^{e}$ can be considered as the quantity  of communication processed by the same vertex. Let us prove that the sum of these quantities is equal.

\begin{theorem}
	\label{Jednake sume}
	For each graph $G$ it holds:%
	\begin{equation*}
	\sum\limits_{u\in V}t_{\lambda }^{e}(u)=\sum\limits_{k,l\in V}d(k,l)\cdot
	\lambda ^{d(k,l)}=\sum\limits_{u\in V}c_{\lambda }^{e}(u).
	\end{equation*}
\end{theorem}

\begin{proof}
	First equality holds by definition of transmission. Next, it holds%
	\begin{equation*}
	\sum\limits_{u\in V}c_{\lambda }^{e}(u)=\sum\limits_{u\in
		V}\sum\limits_{v\in \lbrack u]}\sum\limits_{k,l\in V}\frac{s_{uv}^{kl}}{%
		s^{kl}}\lambda ^{d(k,l)}=\sum\limits_{k,l\in V}\frac{\lambda ^{d(k,l)}}{%
		s^{kl}}\sum\limits_{u\in V}\sum\limits_{v\in \lbrack u]}s_{uv}^{kl}.
	\end{equation*}%
	For a given pair of vertices $(k,l)$, $\underset{u\in V}{\sum }\underset{%
		v\in \lbrack u]}{\sum }s_{uv}^{kl}$ is number of pairs $(u,v)$ such that $%
	d(u,v)=1$ and that a shortest path between $k$ and $l$ passes through the
	edge $uv$. The length of each of the $s^{kl}$ shortest paths from $k$ to $l$
	is $d(k,l)$ and therefore on each such path we can choose $d(k,l)$ pairs $%
	\{u,v\}$ such that $d(u,v)=1$. Thus, 
	\begin{equation*}
	\underset{u\in V}{\sum }\underset{v\in \lbrack u]}{\sum }s_{uv}^{kl}=d(k,l)%
	\cdot s^{kl}.
	\end{equation*}%
	Finally, we have%
	\begin{equation*}
	\sum\limits_{u\in V}c_{\lambda }^{e}(u)=\sum\limits_{k,l\in V}\frac{%
		\lambda ^{d(k,l)}}{s^{kl}}\cdot d(k,l)\cdot s^{kl}=\sum\limits_{k,l\in
		V}d(k,l)\cdot \lambda ^{d(k,l)}.
	\end{equation*}
\end{proof}

\section{Transmission}

Before concentrating on the lower and upper bounds for transmission, we need the following definition. 
\label{Metlica - Definicija}
A \emph{broom}  $B_{n,k}$ is a graph obtained by identification of a pendant vertex
of star $S_{k+1}$ and an end--vertex of path $P_{n-k}$. The other
end--vertex of the path is called \emph{starting vertex} of the broom.
Specially, $B_{n-1,2}=P_{n}$ and $B_{1,n}=S_{n}$. 
For minimal transmission let us prove:
\begin{theorem}
	\label{MIN  Trans donja}
	For each graph $G$ with $n$ vertices it holds%
	\begin{equation}
	\underset{1\leqslant D\leqslant n-1}{\min} {\frac{\lambda \left[ 1 -(D+1) \lambda  ^D+D\lambda ^{D+1} \right]}{(\lambda -1)^2}+(n-D-1)D\cdot \lambda ^{D}} \leqslant mt_{\lambda }^{e}(G)
	\label{min trans}
	\end{equation}%
	The lower bound is reached for a broom (in its starting vertex).
\end{theorem}

\begin{proof}
	Let $G$ be a graph for which the minimum $mt_{\lambda }^{e}(G)$ is attained and let $u$ be a vertex of  the graph for which $t_{\lambda }^{e}(u)=mt_{\lambda }^{e}(G)$. Let $v_D$ be a vertex which is farthest away from $u$ and let $S=uv_1v_2...v_D$ be a shortest path from $u$ to $v_D$. Furthermore, let $k=n-D-1$ and let $W=\{w_1,w_2,...,w_k\}$ be a set of all vertices that do not lie on the path $S$. Since $d(u,v_i)=i$ for all  $i\in \{1,2,...,D\}$ we have: 
	\begin{equation*}
	mt_{\lambda}^{e}(G)=t_{\lambda}^{e}(u)=\sum\limits_{i=1}^{D}i\cdot \lambda ^{i}+\sum\limits_{w \in W}d(u,w)\cdot \lambda ^{d(u,w)}.
	\end{equation*}
	For positive numbers $a_1,a_2,...,a_n$ such that $a=min \{a_1,a_2,...,a_n\}$ it holds 
	\begin{equation*}
	\sum\limits_{i=1}^{n}a_i\geqslant \sum\limits_{i=1}^{n}a=n\cdot a, 
	\end{equation*}
	so
	\begin{equation*}
	mt_{\lambda}^{e}(G)=\sum\limits_{i=1}^{D}i\cdot \lambda ^{i}+\sum\limits_{w \in W}d(u,w)\cdot \lambda ^{d(u,w)}\geqslant \sum\limits_{i=1}^{D}i\cdot \lambda ^{i}+(n-D-1)x\cdot \lambda ^{x}
	\end{equation*}
	where $x=d(u,q)$ for some $q \in W$ for which the expression $d(u,q)\cdot \lambda ^{d(u,q)}$ has minimal value. 
	
	This means that the transmission will be minimal if all the vertices in $W$ are equally distant from $u$.  We will prove that, in that case, $x=D$. Suppose the contrary. Let us observe graph $G'$ which is obtained by removing vertex $v_D$ and connecting it to $v_{x-1}$. Transmission in $G'$ will be smaller than in $G$ which is a contradiction, hence $x=D$. We conclude that one of the graphs for which minimal transmission is attained is indeed a broom, i.e. all the vertices in $W$ are directly connected to $v_{D-1}$. 
\end{proof}

\begin{figure}[H]
	\begin{center}
		\includegraphics[width=\textwidth]{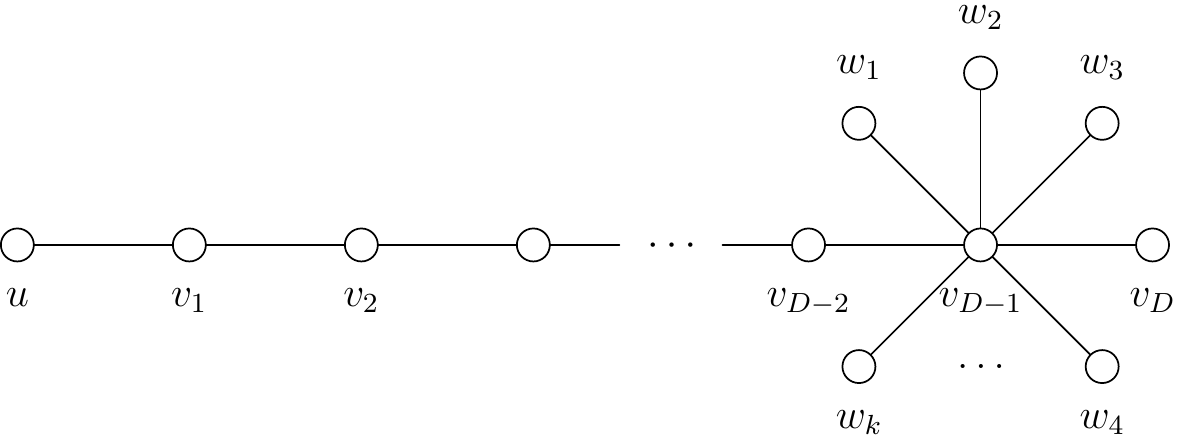}
		\caption{A broom that minimizes $mt_{\lambda }^{e}(G)$.}
		\label{Metlica - Slika}
	\end{center}
\end{figure}

 Under certain assumptions, we can reduce the case $D  \in \left\lbrace 1,..,n-1 \right\rbrace$ to the case when $D \in \left\lbrace1,n-1,\left\lfloor
D_{min}\right\rfloor, \left\lceil D_{min}\right\rceil  \right\rbrace$. Let us prove:

\label{Minimum izraza}
\begin{theorem}
	Function $f\left( D \right)= \sum\limits_{i=1}^{D}i\cdot \lambda ^{i}+(n-D-1)D\cdot \lambda ^{D}$ has local minimum for%
	\begin{equation}
	D_1=\frac{2-2\lambda+ \left[1+\left(\lambda-1 \right)n \right] Log \lambda+ \sqrt{S_{\lambda}}}{2 \left( \lambda-1\right)Log \lambda} 
	\end{equation}%
	and local maximum for%
	\begin{equation}
	D_2=\frac{2-2\lambda+ \left[1+\left(\lambda-1 \right)n \right] Log \lambda- \sqrt{S_{\lambda}}}{2 \left( \lambda-1\right)Log \lambda} 
	\end{equation}%
	where
	$$
	S_{\lambda}=4 \left(\lambda-1 \right)^{2}+\left[ \left( n-1 \right)^{2}+ \lambda^{2}n^{2}-2\lambda\left( n^{2}-n+2\right)\right]Log\lambda^{2},
	$$
	if $D_1, D_2 \in \mathbb{R}$ .
\end{theorem}

\begin{proof}
	The problem reduces to finding minimum (maximum) for the function 
	$$f\left( D \right)= \sum\limits_{i=1}^{D}i\cdot \lambda ^{i}+(n-D-1)D\cdot \lambda ^{D}.$$
	Deriving the function and simplifying it gives us 
	$$
	f^{\prime}\left( D \right) = \frac{\lambda^{D}}{\left( \lambda-1 \right)^{2}}\left( AD^{2}+BD+C\right);
	$$
	where:
	\vskip 0,2cm
	$A= 2\lambda Log\lambda - Log\lambda - \lambda^{2} Log\lambda $;
	\vskip 0,4cm
	
	$B= 4\lambda-2-2\lambda^{2}+Log \lambda \left( \lambda-1+n-2\lambda n+\lambda^{2} n \right) $;
	\vskip 0,4cm
	
	$C=\lambda+n-1-2\lambda n+\lambda^{2} n-\lambda Log\lambda$.
	\vskip 0,4cm
	
	Stationary points are
	
	\begin{equation*}
	D_1=\frac{2-2\lambda+ \left[1+\left(\lambda-1 \right)n \right] Log \lambda+ \sqrt{S_{\lambda}}}{2 \left( \lambda-1\right)Log \lambda} 
	\end{equation*}%
	
	and
	
	\begin{equation*}
	D_2=\frac{2-2\lambda+ \left[1+\left(\lambda-1 \right)n \right] Log \lambda- \sqrt{S_{\lambda}}}{2 \left( \lambda-1\right)Log \lambda} 
	\end{equation*}%
	where
	$$
	S_{\lambda}=4 \left(\lambda-1 \right)^{2}+\left[ \left( n-1 \right)^{2}+ \lambda^{2}n^{2}-2\lambda\left( n^{2}-n+2\right)\right]Log\lambda^{2}.
	$$
	\vskip 0,4cm
	Let us analyze $f^{\prime}\left( D \right)$. Since $\frac{\lambda^{D}}{\left( \lambda-1 \right)^{2}}$ is always positive, whether the function $f\left( D \right)$ is increasing or decreasing depends on the second--degree polynomial. The leading coefficient $2\lambda Log\lambda - Log\lambda - \lambda^{2} Log\lambda > 0$ for $\lambda \in \left\langle 0,1 \right\rangle$.  We conclude that, under the assumption that $D_1, D_2 \in \mathbb{R}$,  the function $f\left( D \right)$ has minimal value for $D_1$ and maximal value for $D_2$.
\end{proof}

\begin{remark}
	\label{Min remark}
	Let us denote $D_{min}=D_1$. If $D_{min} \in \left[1,n-1\right]$ and real, than the minimum of expression $\sum\limits_{i=1}^{D}i\cdot \lambda ^{i}+(n-D-1)D\cdot \lambda ^{D}$  will be reached for some  $D \in \left\lbrace1,n-1,\left\lfloor
	D_{min}\right\rfloor, \left\lceil D_{min}\right\rceil  \right\rbrace$. Otherwise, it will be reached for $D \in \left\lbrace1,n-1  \right\rbrace$. This remark can simplify the calculation of the lower bound in Theorem \ref{MIN  Trans donja}.
\end{remark}

Now, let us concentrate on upper bound. We were able to find it in a special case when $\lambda \in \left\langle 0, \frac{1}{2} \right\rangle$.

\begin{theorem}
	\label{MIN Trans donja - SPECIAL}
	For each graph $G$ with $n$ vertices and for $\lambda \in \left\langle 0, \frac{1}{2} \right\rangle$ it holds%
	\begin{equation*}
	mt_{\lambda }^{e}(G) \leqslant \left( n-1 \right) \cdot \lambda.
	\end{equation*}%
	The lower bound is reached for any vertex of a complete graph.
\end{theorem}

\begin{proof}
	Let $G$ be a graph for which the maximum $mt_{\lambda }^{e}(G)$ is attained and let $u$ be a vertex of  the graph for which $t_{\lambda }^{e}(u)=mt_{\lambda }^{e}(G)$. It holds:%
	\begin{equation*}
	mt_{\lambda }^{e}(G)=t_{\lambda }^{e}(u)=\underset{v\in V\setminus\{u\}}{\sum }d(u,v)\cdot \lambda^{d(u,v)} \leqslant \underset{v\in V\setminus\{u\}}{\sum } \lambda= \left(n-1 \right) \cdot \lambda.
	\end{equation*}%
	
	The inequality holds since for $\lambda \in \left\langle 0, \frac{1}{2} \right\rangle$ function $f \left(x\right)=x \lambda^{x}$ is decreasing. The equality holds for a complete graph since $d \left( u,v \right)=1$ for all $u,v \in V$.
	
\end{proof}

Let us analyse the lower bound for $Mt_{\lambda}^{e}(G)$. We find it  only in the special case of 2--connected graph for $\lambda \in \langle 0, \frac{1}{2}\rangle $. Further we conjecture.

\label{MAX Trans donja - SPECIAL}

\begin{conjecture}
	
	For each graph $G$ with $n \geqslant 3$ vertices and for $\lambda \in \langle 0, \frac{1}{2} \rangle $ it holds that 
	\begin{equation}
	\left\{ 
	\begin{tabular}{c}
	$\frac{\sqrt{\lambda}\left[ 2\sqrt \lambda+ (n-1)\lambda^{1+\frac{n}{2}} -(n+1)\lambda^{\frac{n}{2}} \right]}{(\lambda-1)^{2}}$, $n$\textnormal{ odd} \\ \\
	$\frac{1}{2}n \lambda^{\frac{n}{2}}+\frac{\sqrt{\lambda}\left[ 2\sqrt \lambda+ (n-1)\lambda^{1+\frac{n}{2}} -(n+1)\lambda^{\frac{n}{2}} \right]}{(\lambda-1)^{2}}$, $n$\textnormal{ even}
	\end{tabular}\right\} \leqslant Mt_{\lambda }^{e}(G)
	\label{Ciklus}
	\end{equation}%
	The equality holds for a cycle.
\end{conjecture}

\begin{remark}
	\label{2-povezan}
	The previous conjecture is true in the special case when $G$ is a 2--connected graph. To prove this we need the following lemma.
\end{remark}

\begin{lemma}
	\label{Tanja}
	Let $n \geqslant 3$. Let $\lambda \in \left\langle 0, \frac{1}{2} \right\rangle $ and let $S$ be a set of
	sequences $\left( x_{1},x_{2},...,x_{\left\lfloor n/2\right\rfloor }\right) \in {\mathbb{N}}^{\left\lfloor
		n/2\right\rfloor } $
	such that $x_{1}+x_{2}+...+x_{\left\lfloor n/2\right\rfloor }=n-1$ and there
	exists $k\in \left\{ 1,...,\left\lfloor n/2\right\rfloor \right\} $ such
	that $x_{i}\geq 2$ for each $i\leq k$ and $x_{i}=0$ for each $i>k.$ Let $%
	S^{\prime }$ be the set of sequences in $S$ of the form $\left(
	x_{1},x_{2},...,x_{\left\lfloor n/2\right\rfloor }\right) $ such that there
	is $k\in \left\{ 1,...,\left\lfloor n/2\right\rfloor \right\} $ such that $%
	x_{k}\in \left\{ 0,1\right\} ,$ $x_{i}=2$ for each $1\leq i<k$ and $x_{i}=0$
	for each $i>k.$
	
	Let $T_{n}$ be defined by%
	\[
	T_{n}\left( x_{1},x_{2},...,x_{\left\lfloor n/2\right\rfloor }\right)
	=\sum_{i=1}^{\left\lfloor n/2\right\rfloor }x_{i}\cdot i \cdot \lambda^{i}. 
	\]
	
	Then%
	\[
	\min \left\{ T_{n}\left( s\right) :s\in S\right\} =\min \left\{ T_{n}\left(
	s\right) :s\in S^{\prime }\right\} .
	\]%
	Furthermore, minimal value of $T_{n}$ in $S^{\prime }$ is:%
	
	\begin{equation*}
	\left\{ 
	\begin{tabular}{c}
	$2\cdot \overset{\frac{n-1}{2}}{\underset{i=1}{\sum }}i \cdot \lambda^{i}$, $n$%
	\textnormal{\ odd} \\ 
	$2 \cdot \overset{\frac{n-2}{2}}{\underset{i=1}{\sum }}i \cdot \lambda^{i}+\left( 
	\frac{n}{2}\right)\cdot \lambda ^{\frac{n}{2}}$, $n$\textnormal{\ even}%
	\end{tabular}%
	\right\}
	\end{equation*}%
	
\end{lemma}

\begin{proof}
	Suppose to the contrary. Let $\left( x_{1},x_{2},...,x_{\left\lfloor
		n/2\right\rfloor }\right) \notin S^{\prime }$ minimize $T_{n}$ in $S.$ Then
	there is $k$ such that $x_{k}>2$. Note that $k<\left\lfloor
	n/2\right\rfloor .$ Then, 
	\[
	\left( x_{1},x_{2},...,x_{k}-1,x_{k+1}+1,x_{k+2},...,x_{\left\lfloor
		n/2\right\rfloor }\right) \in S. 
	\]%
	It follows that%
	\begin{eqnarray*}
		0 & \geqslant &T_{n}\left( x_{1},x_{2},...,x_{\left\lfloor n/2\right\rfloor
		}\right) -T_{n}\left(
		x_{1},x_{2},...,x_{k}-1,x_{k+1}+1,x_{k+2},...,x_{\left\lfloor
			n/2\right\rfloor }\right)  \\
		&=&k \lambda^{k}-(k+1) \lambda^{k+1} > 0,
	\end{eqnarray*}%
	which is a contradiction. So the sequence $s\in S^{\prime }$ that minimizes $%
	T_{n}$ is $(2,2,...,2,0)$ for $n$ odd, and $(2,2,...2,1)$ for $n$ even. It
	can be easily seen that the value of $T_{n}$ for those sequences is $2\cdot \overset{\frac{n-1}{2}}{\underset{i=1}{\sum }}i \cdot \lambda^{i}=\frac{\sqrt{\lambda}\left[ 2\sqrt \lambda+ (n-1)\lambda^{1+\frac{n}{2}} -(n+1)\lambda^{\frac{n}{2}} \right]}{(\lambda-1)^{2}}$ in the first case, and $2 \cdot \overset{\frac{n-2}{2}}{\underset{i=1}{\sum }}i \cdot \lambda^{i}+\left( 
	\frac{n}{2}\right)\cdot \lambda ^{\frac{n}{2}}=\frac{1}{2}n \lambda^{\frac{n}{2}}+\frac{\sqrt{\lambda}\left[ 2\sqrt \lambda+ (n-1)\lambda^{1+\frac{n}{2}} -(n+1)\lambda^{\frac{n}{2}} \right]}{(\lambda-1)^{2}}$ in the second case.
\end{proof}

\begin{proof}[Proof of Remark \ref{2-povezan}]
	Let us denote the left--hand side of the inequality \eqref{Ciklus} by $cyc_\lambda(n)$ and assume the contrary--that there exists a 2--connected graph $G$ with $n$ vertices such that $Mt_{\lambda}^{e}(G) < cyc_\lambda(n)$. This implies that $t_{\lambda}^{e}(u) < cyc_\lambda(n)$, for all $u\in V$. Therefore:
	$$
	\sum_{u\in V}t_{\lambda}^{e}(u) < n\cdot cyc_\lambda(n),
	$$
	and thus there exists $w\in V$ such that $t_{\lambda}^{e}(w) < cyc_\lambda(n)$.
	
	Let $w_{1}$ be the vertex that is farthest from $w$ and let $d(w,w_{1})=D.$
	Since $G$ is 2-connected it holds that for every $d<D$ there are at least $2$
	vertices on a distance $d$ from $w$. From that it is easily seen that $D\leq
	\left\lfloor \dfrac{n}{2}\right\rfloor .$ Let us denote with $x_{i}$ the
	number of vertices on a distance $i$ from $w$, and let us observe the
	sequence $(x_{1},...,x_{\left\lfloor n/2\right\rfloor })$. This sequence is
	obviously in $S$ defined in Lemma \ref{Tanja}. It follows $ cyc_\lambda(n) \leqslant t_{\lambda }^{e}(w)$ which is a contradiction.
\end{proof}

For the upper bound let us prove:
\begin{theorem}
	\label{MAX Trans gornja}
	For each graph $G$ with $n$ vertices it holds%
	\begin{equation*}
	Mt_{\lambda }^{e}(G)\leqslant \underset{1\leqslant D\leqslant n-1}{\max} {\frac{\lambda \left[ 1 -(D+1) \lambda  ^D+D\lambda ^{D+1} \right]}{(\lambda -1)^2}+(n-D-1)D\cdot \lambda ^{D}}
	\end{equation*}%
	The equality hold for a broom (in its starting vertex).
\end{theorem}

\begin{proof}
	Let $G$ be a graph for which the maximum $Mt_{\lambda }^{e}(G)$ is attained and let $u$ be a vertex of  the graph for which $t_{\lambda }^{e}(u)=Mt_{\lambda }^{e}(G)$. Let $v_D$ be a vertex which is farthest away from $u$ and let $S=uv_1v_2...v_D$ be a shortest path from $u$ to $v_D$. Furthermore, let $k=n-D-1$ and let $W=\{w_1,w_2,...,w_k\}$ be a set of all vertices that do not lie on the path $S$. Since $d(u,v_i)=i$ for all  $i\in \{1,2,...,D\}$ we have: 
	\begin{equation*}
	Mt_{\lambda}^{e}(G)=t_{\lambda}^{e}(u)=\sum\limits_{i=1}^{D}i\cdot \lambda ^{i}+\sum\limits_{w \in W}d(u,w)\cdot \lambda ^{d(u,w)}.
	\end{equation*}
	For positive numbers $a_1,a_2,...,a_n$ such that $a=max \{a_1,a_2,...,a_n\}$ it holds 
	\begin{equation*}
	\sum\limits_{i=1}^{n}a_i\leqslant \sum\limits_{i=1}^{n}a=n\cdot a, 
	\end{equation*}
	so
	\begin{equation*}
	Mt_{\lambda}^{e}(G)=\sum\limits_{i=1}^{D}i\cdot \lambda ^{i}+\sum\limits_{w \in W}d(u,w)\cdot \lambda ^{d(u,w)}\leqslant \sum\limits_{i=1}^{D}i\cdot \lambda ^{i}+(n-D-1)x\cdot \lambda ^{x}
	\end{equation*}
	where $x=d(u,q)$ for some $q \in W$ for which the expression $d(u,q)\cdot \lambda ^{d(u,q)}$ has maximal value.
	
	This means that the transmission will be maximal if all the vertices in $W$ are equally distant from $u$. We will prove that, in that case, $x=D$. Suppose that is not the case. Let us observe graph $G'$ which is obtained by removing vertex $v_D$ and connecting it to $v_{x-1}$. Graph  $G'$ will have larger transmission than $G$ which is a contradiction. We conclude that one of the graphs for which the transmission is maximal is a broom, i.e., all the vertices in $W$ are directly connected to $v_{D-1}$.
	
\end{proof}

\begin{remark}
	\label{max remark}
	Let us denote $D_2$ in Theorem \ref{Minimum izraza} as $D_{max}$.  If $D_{max} \in \left[ 1,n-1\right]$ and real, than the maximum of expression $\sum\limits_{i=1}^{D}i\cdot \lambda ^{i}+(n-D-1)D\cdot \lambda ^{D}$  will be reached for some  $D \in \left\lbrace1,n-1,\left\lfloor
	D_{max}\right\rfloor, \left\lceil D_{max}\right\rceil  \right\rbrace$. Otherwise, it will be reached for $ D \in \left\lbrace1,n-1 \right\rbrace$. This remark can simplify the calculation of the lower bound in Theorem \ref{MAX Trans gornja}.
	
\end{remark}

\section{Betweenness Centrality}

\begin{lemma}
	\label{Max Cent je stablo}
	For all $\lambda \in \left\langle 0,1 \right\rangle$ and for a given integer $n$, among all graphs with $n$ vertices, any graph $G$ for which maximum $c_{\lambda}^{e}(G)$ is obtained is a tree.
\end{lemma}

\begin{proof}
	Let $G$ be a graph such that $c_{\lambda }^{e}(G)$ is maximal and let $u$ be
	a vertex for witch maximal centrality is reached. We will prove that $G$
	is a tree.
	
	Suppose that is not the case. Let us observe Dijkstra spanning tree $G^{\prime } $ that
	is obtained as follows. Starting from vertex $u$, in each step we choose a
	vertex $v$ that is closest to $u$ (the distance between $u$ and $v$ is
	minimal) and is still outside the tree. Since $G^{\prime } $ is a tree, it
	holds that $\frac{s_{uv}^{kl}}{s^{kl}}=1$ for each $k,l\in V$ \ that are
	connected by a path passing through the edge $uv.$ From the way $G^{\prime }$
	was obtained, it is obvious that the distances between $u$ and $v$, for
	every $v\in V$, will stay the same. This means that $c_{\lambda }^{e}(u)$
	is greater in $G^{\prime }$ than in $G$ which contradicts our assumption.
\end{proof}

\begin{lemma}
	\label{Cent za put}
	For each graph $G$ with $n$ vertices and for $\lambda \in \left\langle0,1 \right\rangle$ it holds%
	\begin{equation*}
	\underset{v\in V\setminus\{u\}}{\sum}\lambda^{d(u,v)}\geqslant\sum\limits_{i=1}^{n-1}\lambda ^{i}=\frac{\lambda^D-\lambda}{\lambda -1}.
	\end{equation*}%
\end{lemma}

\begin{proof}
	Let $G$ be a graph with $n$ vertices and let $u$ and $v$ be two vertices which are connected by the longest path in a graph, i.e. $d(u,v)=diam(G)$.  Let $W$ be set of all vertices that don't lie on path from $u$ to $v$. Let us consider graph $G'$ which is obtaind by cutting any of the vertices $w\in W$ and putting it on $v$.  By doing this we increased distances beetween vertices, and thus, since $\lambda \in \left\langle0,1 \right\rangle$, we decreased the sum. Continuing this process leads us to the conclusion that the wanted sum will be minimal if graph $G$ is a path, i.e. it holds%
	\begin{equation*}
	\underset{v\in V\setminus\{u\}}{\sum}\lambda^{d(u,v)}\geqslant\sum\limits_{i=1}^{n-1}\lambda ^{i}=\frac{\lambda^D-\lambda}{\lambda -1}.
	\end{equation*}%
	
\end{proof}
\begin{theorem}
	\label{MIN Cent donja}
	For each graph $G$ with $n$ vertices it holds%
	\begin{equation*}
	\frac{\lambda^D-\lambda}{\lambda -1}\leqslant mc_{\lambda }^{e}(G).
	\end{equation*}%
	The lower bound is reached for a path (in its end--vertex).
\end{theorem}

\begin{proof}
	Let us prove the lower bound. Let $G$ be a graph for which $mc_{\lambda
	}^{e}(G)$ is minimal and let $u$ be a vertex such that $c_{\lambda
	}^{e}(u)=mc_{\lambda }^{e}(G)$. Using Lemma \ref{Cent za put}, it holds%
	\begin{equation*}
	mc_{\lambda }^{e}(G)\geqslant\underset{v\in V\setminus\{u\}}{\sum}\lambda^{d(u,v)}\geqslant\sum\limits_{i=1}^{n-1}\lambda ^{i}=\frac{\lambda^D-\lambda}{\lambda -1}.
	\end{equation*}%
\end{proof}

For the upper bound we solve the problem for $\lambda \in \left\langle 0, \frac{1}{2} \right\rangle$.

\begin{theorem}
	
	\label{MIN Cent gornja - SPECIAL}
	For each graph $G$ with $n$ vertices and for $\lambda \in \left\langle 0, \frac{1}{2} \right\rangle$  it holds%
	\begin{equation*}
	mc_{\lambda }^{e}(G)\leqslant (n-1)\cdot \lambda
	\end{equation*}%
	The upper bound is reached for a complete graph (in any of its vertices).
	
\end{theorem}

\begin{proof}
	Using Theorem \ref{Jednake sume}, we can bound the average
	centrality of all vertices in the following way:%
	\begin{equation*}
	\frac{1}{n}\sum\limits_{u\in V}c_{\lambda }^{e}(u)=\frac{1}{n}%
	\sum\limits_{k,l\in V}d(k,l)\cdot \lambda ^{d(k,l)}\leqslant \frac{1}{n}%
	\sum\limits_{k,l\in V}\lambda =\frac{1}{n}n(n-1)\cdot \lambda =(n-1)\cdot
	\lambda. 
	\end{equation*}%
	Since minimal centrality is smaller than or equal to the average centrality,
	the claim is proven. The equality holds for a complete graph since $d(k,l)=1$
	for any two vertices $k,l\in V.$
	
\end{proof}

\begin{theorem}
	\label{MAX Cent gornja}
	For each graph $G$ with $n$ vertices it holds%
	\begin{equation*}
	Mc_{\lambda }^{e}(G)\leqslant (n-1)[\lambda +\frac{1}{2}(n-2)\cdot \lambda ^{2}].
	\end{equation*}%
	The equality holds for a star (in its central vertex).
\end{theorem}

\begin{proof}
	Using Lemma \ref{Max Cent je stablo} we conclude that the wanted graph is a tree. Let $G$ be a tree such that $c_{\lambda }^{e}(G)$ is maximal and let $u$ be a vertex for which maximal centrality is reached. Let $P$ be a set of all unordered pairs of vertices $v,w \in V \setminus \{u\}$ such that the shortest path from $v$ to $w$ passes through $u$. It holds:%
	\begin{eqnarray*}
		Mc_{\lambda }^{e}(G) &=&c_{\lambda }^{e}(u)=\sum\limits_{v\in \lbrack
			u]}\sum\limits_{k,l\in V}\frac{s_{uv}^{kl}}{s^{kl}}\lambda
		^{d(k,l)} = \sum_{v \in V\setminus \{u\}}\lambda
		^{d(u,v)}+\sum_{\left\lbrace v,w \right\rbrace \in P}\lambda ^{d(v,w)}\\
		&\leqslant &  (n-1)\cdot
		\lambda + \frac{1}{2}(n-1)(n-2)\cdot \lambda ^{2} \\
		&= &(n-1)[\lambda +\frac{1}{2}(n-2)\cdot \lambda ^{2}].
	\end{eqnarray*}
	Maximal centrality is reached for the central vertex of a star since all vertices $v \in V \setminus\{u\}$ are directly connected to $u$ and for all vertices $v,w \in V \setminus\{u\}$ holds $d(v,w)=2$.
	
\end{proof}

\section{Networkness}
In paper \cite{BV} it has been proven that:

\begin{lemma}
	\label{Raspisano} 
	For positive numbers $a_{1},a_{2},\ldots
	,a_{n},b_{1},b_{2},\ldots ,b_{n}$ the following holds: 
	\begin{equation*}
	\frac{\overset{n}{\underset{i=1}{\sum }}{a_{i}}}{\overset{n}{\underset{i=1}{%
				\sum }}{b_{i}}}\geqslant \min \{\frac{a_{i}}{b_{i}}\}.
	\end{equation*}
\end{lemma}
Using this, let us prove:

\begin{theorem}
	\label{MIN Net obe}
	For each graph $G$ with $n$ vertices it holds%
	\begin{equation*}
	\underset{2\leqslant D\leqslant n-1}{\min} {\frac{\lambda^D+\frac{1}{n-D}\left(\frac{\lambda^D-\lambda}{\lambda -1}\right)}{D\lambda^D+\frac{1}{n-D}\left[\frac{\lambda -D \lambda  ^D+(D-1)\lambda ^{D+1}}{(\lambda -1)^2}\right]}}\leqslant mN_{\lambda }^{e}(G)\leqslant 1.
	\end{equation*}%
	The lower bound is reached for a broom ( in its starting vertex) and the
	upper bound is reached for any vertex of a vertex- transitive graph.
\end{theorem}

\begin{proof}
	Using Theorem \ref{Jednake sume} and Lemma \ref{Raspisano}, it holds%
	\begin{equation*}
	min \left\{ \frac{c_{\lambda }^{e}(u)}{t_{\lambda }^{e}(u)}:u\in V\right\}
	\leqslant \frac{\sum\limits_{u\in V}c_{\lambda }^{e}(u)}{%
		\sum\limits_{u\in V}t_{\lambda }^{e}(u)}=\frac{\sum\limits_{k,l\in
			V}d(k,l)\cdot \lambda ^{d(k,l)}}{\sum\limits_{k,l\in V}d(k,l)\cdot \lambda
		^{d(k,l)}}=1.
	\end{equation*}%
	Let us prove the lower bound. Let $G$ be a graph for which the minimum of $%
	mN_{\lambda }^{e}(G)$ is attained and let $u$ be a vertex of the graph for which $
	N_{\lambda }^{e}(u)=mN_{\lambda }^{e}(G) $. It holds

	\begin{equation*}
	N_{\lambda }^{e}(G)=\frac{c_{\lambda }^{e}(G)}{t_{\lambda }^{e}(G)}\geqslant
	\frac{\sum\limits_{v\in V\setminus \left\{ u\right\} }\lambda ^{d(u,v)}}{%
		\sum\limits_{v\in V\setminus \left\{ u\right\} }d(u,v)\lambda ^{d(u,v)}}
	\end{equation*}%
	because $u$ certainly lies on every shortest path between itself and every
	other vertex $v$. Now let $v_{D}$ be a vertex which is farthest away from $u$
	and let $S=uv_{1}v_{2}...v_{D\text{ }}$be a shortest path from $u$ to $v_{D}$%
	. Furthermore, let $k=n-D-1$, let $\{w_{1},...,w_{k}\}=V\backslash
	\{u,v_{1},...,v_{D}\}$ be set of all vertices that do not lie on the path $S$
	and let $W=\{w_{1},w_{2},...,w_{k},v_{D}\}.$

	Because $d(u,v_{i})=i$ for all $i\in \{1,2,...,D\}$, we have:%
	
	\begin{equation}
	\label{Netw1}
	N_{\lambda }^{e}(u)\geqslant \frac{\sum\limits_{v\in
			V\setminus \left\{ u\right\} }\lambda ^{d(u,v)}}{\sum\limits_{v\in
			V\setminus \left\{ u\right\} }d(u,v)\lambda ^{d(u,v)}}=\frac{%
		\sum\limits_{i=1}^{D}\lambda ^{i}+\sum\limits_{i=1}^{k}\lambda
		^{d(u,w_{i})}}{\sum\limits_{i=1}^{D}i\lambda
		^{i}+\sum\limits_{i=1}^{k}d(u,w_{i})\lambda ^{d(u,w_{i})}}.
	\end{equation}%
	
	The last expression in \eqref{Netw1} can be written as%
	\begin{equation}\label{Netw2}
	\frac{\sum\limits_{v\in W}\left( \lambda ^{d(u,v)}+\frac{1}{n-D}%
		\sum\limits_{i=1}^{D-1}\lambda ^{i}\right) }{\sum\limits_{v\in W}\left(
		d(u,v)\lambda ^{d(u,v)}+\frac{1}{n-D}\sum\limits_{i=1}^{D-1}i\lambda
		^{i}\right) }.
	\end{equation}%
	
	Using Lemma \ref{Raspisano}, the minimum of expression \eqref{Netw2} is 
	\begin{equation*}
	\frac{\lambda ^{x}+\frac{1}{n-D}\sum\limits_{i=1}^{D-1}\lambda ^{i}}{%
		x\lambda ^{x}+\frac{1}{n-D}\sum\limits_{i=1}^{D-1}i\lambda ^{i}},
	\end{equation*}%
	where $x=d(u,q)$ for some $q\in W$ for which the expression \eqref{Netw2} has minimal value. \newline
	This minimum is obtained if and only if ratio $\frac{a_{i}}{b_{i}}$ is
	constant for all $i\in \{1,2,...,n\}$ and one way to achieve this is that $%
	d(u,v)$ is constant for all $v\in W$, i.e. that $d(u,w_{i})$ $=d(u,v_{D})=D$
	for all $i\in \left\{ 1,2,...,k\right\} .$ This is possible if $w_{i}$
	is directly connected to $v_{D-1}$ for all $i\leqslant k$, which is true when $G$
	is a broom.
\end{proof}
\begin{theorem}
	\label{MAX Net obe}
	For each graph $G$ with $n$ vertices it holds%
	\begin{equation*}
	1\leqslant MN_{\lambda }^{e}(G)\leqslant \frac{1}{2}\left( n-2\right) \cdot \lambda +1.
	\end{equation*}%
	The lower bound is reached for any vertex of a vertex-transitive graph and the upper bound is reached for a star (in its central vertex).
\end{theorem}

\begin{proof}
	First, let us prove the lower bound. Using Theorem \ref{Jednake sume}, since maximum is greater than or equal to average, we have%
	\begin{equation*}
	\max \left\{ \frac{c_{\lambda }^{e}(u)}{t_{\lambda }^{e}(u)}:u\in V\right\}
	\geqslant \frac{\frac{1}{n}\sum\limits_{u\in V}c_{\lambda }^{e}(u)}{\frac{1}{n}%
		\sum\limits_{u\in V}t_{\lambda }^{e}(u)}=\frac{\sum\limits_{k,l\in
			V}d(k,l)\cdot \lambda ^{d(k,l)}}{\sum\limits_{k,l\in V}d(k,l)\cdot \lambda
		^{d(k,l)}}=1.
	\end{equation*}%
	Now, let us prove the upper bound. From Lemma \ref{Max Cent je stablo} we conclude that the graph
	that maximizes $MN_{\lambda }^{e}(G)$ is a tree. 
	Namely, since networkness is defined as quotient of betweenness centrality and transmission, we want the numerator to be maximal. That holds when the graph $G$ is a tree. We can assume that the graph used in the denominator is also a tree. If this is not the case, we can repeat the construction of $G^\prime$ in Lemma \ref{Max Cent je stablo} to obtain a tree in which the distances between  $u$ and all the other vertices stay the same, thus, transmission stays the same.

	Let $u\in V$ be a vertex
	that maximizes networkness. Let $P$ be a set of all unordered pairs of vertices $v,w \in V\setminus \left\lbrace u \right\rbrace$ such that the shortest path between $v$ and $w$ passes through $u$. It holds:%
	
	\begin{eqnarray*}
		N_{\lambda }^{e}(u) &=&\frac{\sum\limits_{k,l\in V}\lambda ^{d(k,l)}}{%
			\sum\limits_{v\in V\setminus \left\{ u\right\} }d(u,v)\cdot \lambda
			^{d(u,v)}}\leqslant \frac{\sum\limits_{
				\left\lbrace v,w \right\rbrace \in P}\lambda ^{\left[
				d(v,u)+d(u,w)\right] }+\sum\limits_{v\in V\setminus \{u\}}\lambda ^{d(u,v)}%
		}{\sum\limits_{v\in V\setminus \left\{ u\right\} }d(u,v)\cdot \lambda
			^{d(u,v)}} \\
		&=&\frac{\sum\limits_{\left\lbrace v,w \right\rbrace 
				\in P}\lambda ^{d(v,u)}\cdot
			\lambda ^{d(u,w)}+\sum\limits_{v\in V\setminus \{u\}}\lambda ^{d(u,v)}}{%
			\sum\limits_{v\in V\setminus \left\{ u\right\} }d(u,v)\cdot \lambda
			^{d(u,v)}} \\
		&\leqslant &\frac{\frac{1}{2}(n-2)\cdot \lambda \cdot \sum\limits_{v\in V\setminus
				\{u\}}\lambda ^{d(v,u)}+\sum\limits_{v\in V\setminus \{u\}}\lambda ^{d(u,v)}%
		}{\sum\limits_{v\in V\setminus \left\{ u\right\} }d(u,v)\cdot \lambda
			^{d(u,v)}} \\
		&\leqslant &\frac{\left[ \frac{1}{2}(n-2)\cdot \lambda +1\right] \sum\limits_{v\in
				V\setminus \{u\}}\lambda ^{d(u,v)}}{\sum\limits_{v\in V\setminus \left\{
				u\right\} }d(u,v)\cdot \lambda ^{d(u,v)}}\leqslant \frac{1}{2}(n-2)\cdot \lambda +1.
	\end{eqnarray*}%
	\\
	Simple calculation show that equality holds for a central vertex of a star.
\end{proof}

\section{Network Surplus}

\begin{theorem}
	\label{MIN Surplus obe}
	For each graph $G$ with $n$ vertices it holds%
	\begin{equation*}
	\min_{1 \leqslant D\leqslant n-1}{\frac{\lambda \left[ D \lambda^{D}-\lambda-(D-1)\lambda ^{D+1} \right]}{(\lambda -1)^2} +(n-D-1)(\lambda^D- D \lambda ^{D})}\leqslant m\nu _{\lambda }^{e}(G)
	\end{equation*}%
	and%
	\begin{equation*}
	m\nu _{\lambda }^{e}(G) \leqslant 0.
	\end{equation*}
	The lower bound is reached for broom (in its starting vertex) and the upper bound is reached for  a vertex-transitive graph (in any of its vertices).
\end{theorem}

\begin{proof}
	Let us prove the upper bound. Using Theorem \ref{Jednake sume}, we have:%
	\begin{eqnarray*}
		m\nu _{\lambda }^{e}(u) &=&\min \left\{ c_{\lambda }^{e}(u)-t_{\lambda
		}^{e}(u):u\in V\right\} \leqslant \frac{1}{n}\left( \sum\limits_{u\in V}\left(
		c_{\lambda }^{e}(u)-t_{\lambda }^{e}(u)\right) \right) \\
		&=&\frac{1}{n}\left( \sum\limits_{u\in V}c_{\lambda
		}^{e}(u)-\sum\limits_{u\in V}t_{\lambda }^{e}(u)\right)=0.
	\end{eqnarray*}
	The first inequality holds since minimum is smaller than or equal to the average.
	For the lower bound, let us suppose $G$ is a graph for which the minimum $m\nu_{\lambda }^{e}(G)$ is attained an let $u$ be a vertex of  the graph for which the minimum is attained. Let $v_D$ be a vertex which is farthest away from $u$ and let $S=uv_1v_1...v_d$ be a shortest path from $u$ to $v_D$. Furthermore, let $k=n-D-1$ and let $W=\{w_1,w_2,...,w_k\}$ be a set of all vertices that do not lie on the path $S$. Since $d(u,v_i)=i$ for all  $i\in \{1,2,...,D\}$ we have: 
	\begin{eqnarray*}
		m\nu _{\lambda }^{e}(u) &=&c_{\lambda }^{e}(u)-t_{\lambda }^{e}(u)\\
		&\geqslant&\sum\limits_{i=1}^{D}\lambda ^{i}+\sum\limits_{w\in W}\lambda
		^{d(u,w)}- \sum\limits_{i=1}^{D}i \cdot \lambda ^{i}-\sum\limits_{w\in W}d(u,w)\lambda^{d(u,w)} \\
		&\geqslant &\sum\limits_{i=1}^{D}\left( \lambda ^{i}- i\cdot \lambda^{i}\right)+\sum\limits_{w\in W}\lambda
		^{d(u,w)}[1-d(u,w)]
		\\
		&\geqslant &\sum\limits_{i=1}^{D}\left( \lambda ^{i}- i\cdot \lambda^{i}\right)+(n-D-1)\lambda^{x}(1-x).
	\end{eqnarray*}
	where $x=d(u,q)$ for some $q\in W$ which minimizes the expression $\lambda
	^{d(u,q)}[1-d(u,q)]$.
	
	This means that all the vertices in $W$ are equally away form $u$. As proven in Theorem \ref{MAX Trans gornja}, in this case it holds $d(u,w)=D$ for all $w\in W$, i.e., all the vertices in $W$ are directly connected to $v_{D-1}$. We conclude that one of the graphs for which the lower bound is reached is a broom.
\end{proof}

\begin{theorem}
	\label{MAX Surplus obe}
	For each graph $G$ with $n$ vertices it holds%
	\begin{equation*}
	0\leqslant M\nu _{\lambda }^{e}(G)\leqslant \frac{1}{2}(n-1)(n-2)\cdot \lambda ^{2}.
	\end{equation*}%
	The lower bound is reached for any vertex a vertex-transitive graph and the upper
	bound is reached for a  a star (in its central vertex).
\end{theorem}

\begin{proof}
	For lower bound, we have, again using Theorem \ref{Jednake sume}%
	\begin{eqnarray*}
		M\nu _{\lambda }^{e}(u) &=&\max \left\{ c_{\lambda }^{e}(u)-t_{\lambda
		}^{e}(u):u\in V\right\}  \geqslant \frac{1}{n}\left( \sum\limits_{u\in V}\left(
		c_{\lambda }^{e}(u)-t_{\lambda }^{e}(u)\right) \right)  \\
		&=&\frac{1}{n}\left( \sum\limits_{u\in V}c_{\lambda
		}^{e}(u)-\sum\limits_{u\in V}t_{\lambda }^{e}(u)\right)=0.
	\end{eqnarray*}%
	The first ineqality holds because maximum is greater than or equal to the average.
	Now, let us prove the upper bound. From Lemma \ref{Max Cent je stablo} it is obvious that the
	wanted graph is a tree. Let $G$ be a graph for witch $M\nu _{\lambda }^{e}(G)
	$ is maximal and let $u\in V$ be a vertex such that $\nu _{\lambda
	}^{e}(u)=M\nu _{\lambda }^{e}(G)$.  Let $P$ be a set of all unordered pairs of vertices $v,w \in V \setminus\{u\}$ such that the shortest path between $v$ and $w$ passes through $u$. It holds:%
	\begin{eqnarray*}
		m\nu _{\lambda }^{e}(u) &=&c_{\lambda }^{e}(u)-t_{\lambda }^{e}(u)\leqslant
		\sum\limits_{k,l\in V}\lambda ^{d(k,l)}-\sum\limits_{v\in V\setminus
			\left\{ u\right\} }d(u,v)\cdot \lambda ^{d(u,v)} \\
		&\leqslant &\sum\limits_{
			\left\lbrace  v,w \right\rbrace \in P}\lambda ^{\left[ d(v,u)+d(u,w)%
			\right] }+\sum\limits_{v\in V\setminus \{u\}}\lambda
		^{d(u,v)}-\sum\limits_{v\in V\setminus \left\{ u\right\} }d(u,v)\cdot
		\lambda ^{d(u,v)} \\
		&\leqslant &\sum\limits_{
			\left\lbrace v,w \right\rbrace \in P}\lambda ^{2}+\sum\limits_{v\in
			V\setminus \{u\}}\lambda ^{d(u,v)}-\sum\limits_{v\in V\setminus \left\{
			u\right\} }d(u,v)\cdot \lambda ^{d(u,v)} \\
		&\leqslant &\frac{1}{2}(n-1)(n-2)\cdot \lambda ^{2}+\sum\limits_{v\in V\setminus
			\{u\}}\lambda ^{d(u,v)}-\sum\limits_{v\in V\setminus \left\{ u\right\}
		}d(u,v)\cdot \lambda ^{d(u,v)} \\
		&\leqslant &\frac{1}{2}(n-1)(n-2)\cdot \lambda ^{2}.
	\end{eqnarray*}
\end{proof}

\section*{Discussion and Conclusions}
Transmission and betweenness centrality are well known concepts in communication networks theory. Based on them, new concepts of networkness an network surplus have been defined \cite{BI8}. They include the assumption  of equal communcation between vertices. Based on a new assumption that communication decreases as the distance between vertices increases, generalised network descriptors were presented. In \cite{BV} the amount of communication was weighted by $d(u,v)^\lambda$ where $\lambda<0$. In this paper we wanted to explore a more radical  assumption, so we weighted the amount of communication  by $\lambda^{d(u,v)}$ where $\lambda \in \left\langle0,1 \right\rangle$.
We have defined and analyzed exponential generalised network descriptors. Extremal values of these descriptors and graphs which they are obtained for can be found in Table 1. Lower bounds for $Mt_{\lambda }^{e}(G)$ and $Mc_{\lambda }^{e}(G)$ remain an open problems.

\section{Acknowledgements}
This work was partially supported by the Croatian Ministry of Science, Education and Sports (Grant NOS 177-0000000-0884 and 037-0000000-2779) and GReGAS within Euro GIGA collaborative research project is gratefully acknowledged.

\end{document}